\documentclass[11pt,a4paper,draft]{amsart}
\usepackage{amsfonts,amsmath,amssymb,amsthm}
\usepackage{times}
\usepackage{color}

\numberwithin{equation}{section}

\renewcommand{\epsilon}{\varepsilon}

\renewcommand{\Re}{{\ensuremath{\mathrm{Re\,}}}}

\DeclareSymbolFont{SY}{U}{psy}{m}{n}
\DeclareMathSymbol{\emptyset}{\mathord}{SY}{'306}

\DeclareMathSymbol{\newtimes}{\mathbin}{SY}{'264}

\newcommand{\R}{\mathbb{R}}

\newcommand{\N}{\mathbb{N}}

\newcommand{\supp}{\mathrm{supp}}

\newcommand{\e}{\mathrm{e}}

\newtheorem{introtheorem}{Theorem}{\bf}{\it}

\newtheorem{theorem}{Theorem}[section]{\bf}{\it}
{\bf}{\it}
{\bf}{\it}
{\bf}{\it}
{\it}{\rm}
\newtheorem{lemma}[theorem]{Lemma}{\bf}{\it}
\newtheorem{remark}[theorem]{Remark}{\it}{\rm}
{\bf}{\it}
{\bf}{\it}
{\bf}{\it}

\newtheorem{lem}{Lemma}{\bf}{\it}

\newtheorem{assumption}{Assumption}

\title[Unstable Bumps in Neural Networks]{On the Existence of Unstable Bumps in Neural Networks}

\author[V. Kostrykin]{Vadim Kostrykin}
\address{V.~Kostrykin, FB 08 - Institut f\"{u}r Mathematik,
Johannes Gutenberg-Universit\"{a}t Mainz,
Staudinger Weg 9,
D-55099 Mainz,
Germany}
\email{kostrykin@mathematik.uni-mainz.de}

\author[A.~Oleynik]{Anna Oleynik}
\address{A.~Oleynik, Department of Mathematics, University of Uppsala, Box 480, S-75106 Uppsala, Sweden}
\email{anna.oleynik@math.uu.se}



\begin{document}

\maketitle

\begin{abstract}
We study the neuronal field equation, a nonlinear integro-differential equation of Hammerstein type.
By means of the Amann three fixed point theorem we prove the existence of bump solutions to this
equation. Using the Krein-Rutman theorem we show their Lyapunov instability.
\end{abstract}

\footnotetext{This work is supported in part by the Deutsche Forschungsgemeinschaft, grant KO 2936/4-1.}

\section{Introduction and Main Results}

The behavior of a single layer of neurons can be modeled by a nonlinear integro-differential equation of the Hammerstein type,
\begin{equation}\label{eq:1}
 \partial_t u(x,t) = -u(x,t) + \int_{\R} \omega(x-y) f(u(y,t)-h) dy.
\end{equation}
Here $u(x,t)$ and $f(u(x,t)-h)$ represent the averaged local activity and the firing rate of neurons at the position $x\in \R$ and time $t>0$, respectively. The parameter $h \geq 0$ is a firing threshold, and $\omega(x-y)$ describes a coupling between neurons at positions $x$ and $y$.

The model described above has been studied in numerous mathematical papers (for a review see, e.g., \cite{Coombes}, \cite{Ermentrout}). In particular, the global existence and uniqueness of solutions to the initial value problem for equation \eqref{eq:1} under rather mild assumptions on $f$ and $w$ has been proven in \cite{Potthast:Graben}.

In 1977, Amari studied pattern formation in \eqref{eq:1} for a model where $f$ is the Heaviside function and
$\omega$ is assumed to be continuous, integrable and even, with $\omega(0)>0$ and having exactly one positive zero.
In particular, he showed the existence of stable and unstable bumps, that is, time independent spatially localized solutions to \eqref{eq:1}. For more general $f$ and $\omega$ the existence of stable solutions of this kind has been shown by Kishimoto and Amari in \cite{Kishimoto:Amari} and later generalized by Oleynik, Ponosov, and Wyller in \cite{A}. In the present work we prove the existence of unstable bumps.


Our main assumptions are as follows.

\begin{assumption}\label{ass:f}
 Let $f:\R\to [0,1]$ be an arbitrary continuous nondecreasing function such that $f(x)=0$ for all $x\leq 0$ and $f(x)=1$ for all $x\geq \tau$ with $\tau>0$.
\end{assumption}

In particular, $f$ is a distribution function of a continuous probability measure supported on the interval $[0,\tau]$. As an example of such function we have
\begin{equation} \label{eq:f:ex}
f(u)=\begin{cases} 0, & u\leq 0,\\ \displaystyle\frac{u^p}{u^p + (\tau - u)^p}, & 0<u< \tau,\\ 1, & u\geq \tau, \end{cases}
\end{equation}
with $p>0$ arbitrary. It is straightforward to see that $f\in C^{\lfloor p\rfloor}(\R)$, where $\lfloor p\rfloor$ denotes the integer part of $p$.

\bigskip

\begin{assumption}\label{ass:w}
We assume that  the integral kernel $\omega$ meets the following conditions:
\begin{itemize}
\item[(i)]$\int_{-\infty}^{\infty} |\omega(x)| dx < \infty$, that is,  $\omega \in L^1(\R)$.
\item[(ii)] $\omega$ is bounded and continuous.
\item[(iii)] $\omega$ is a symmetric function, i.e., $\omega(-x)=\omega(x)$.
\item[(iv)] There is an $a>0$ such that $\omega(x)> 0$ for almost all $x\in [0,2a]$.
\item[(v)] For given $h,\tau>0$
\begin{equation*}
\int_0^{2a} \omega(y) dy > h+\tau.
\end{equation*}
\end{itemize}
The conditions (i)-(v) guarantee that there are $0 < \Delta_- < \Delta_+<a$ such that
\begin{equation*}
\int_{0}^{2\Delta_-} \omega(y)dy = h,\qquad \int_{0}^{2\Delta_+} \omega(y)dy = h + \tau.
\end{equation*}
For all $x\in\R$ we define
\begin{equation}\label{eq:upm}
u_\pm(x) := \int_{-\Delta_\pm}^{\Delta_\pm} \omega(x-y) dy.
\end{equation}
\begin{itemize}
\item[(vi)] There is a $d\in(\Delta_+, a]$ such that $u_+(d) = h$.
\item[(vii)] $\omega$ is decreasing on $[0,2d]$ and $\omega(x)\leq \omega(2d)$ for all $x\geq 2d$.
\end{itemize}
\end{assumption}

Let  $\chi_{(\tau,\infty)}$ and $\chi_{(0,\infty)}$ be characteristic functions of $(\tau,\infty)$ and $(0,\infty)$, respectively. Under Assumption \ref{ass:w} it is easy to show that the functions $u_+$ and $u_-$ solve  equation \eqref{eq:1} with $f=\chi_{(\tau,\infty)}$ and $f=\chi_{(0,\infty)}$, respectively. The proof is given in Appendix, see Lemma \ref{lem:app:1}.

Following Amari \cite{Amari} we call a stationary solution of equation \eqref{eq:1} a \emph{bump} (more precisely, 1-bump) if the support of the function $x\mapsto f(u(x)-h)$ is an interval. According to this definition $u_+$ and $u_-$ are bumps provided $f=\chi_{(\tau,\infty)}$ and $f=\chi_{(0,\infty)}$, respectively, see Lemma \ref{lem:app:1} in Appendix.

One of the common choices of $\omega$ in the study of
neural field models is that of a 'Mexican hat' function,
such as
\begin{equation*}
\omega(x) = K\exp(-k x^2) - M\exp(-m x^2),\quad K>M>0,\quad k>m>0,
\end{equation*}
see, e.g., \cite{A}, \cite{Coombes}, \cite{Amari}.
This function satisfies Assumption B for some values of $h$ and $\tau$.  The other common choices of $\omega$ are the exponential function
$\omega(x)=\e^{-|x|}/2$ and the Gaussian function $\omega(x)=\exp(-x^2)$. It is easy to see that the conditions of Assumption B are satisfied for these functions if $h+\tau<1/2$ and $h+\tau<\sqrt{\pi}/2$, respectively.

The condition (ii) of Assumption \ref{ass:w} implies that $u_\pm$ are continuous, whereas from the conditions (iii) and (iv) the inequality $u_-(x)<u_+(x)$ for all $x\in [-d,d]$ follows.

\begin{lemma}\label{lem:blin}
The condition (vii) in Assumption \ref{ass:w} is fulfilled if and only if
\begin{equation}\label{eq:cond}
 \omega(x-y) \leq \omega(d-y)\quad\text{for all}\quad x>d\quad\text{and}\quad y\in [-d,d].
\end{equation}
\end{lemma}

\begin{proof}
Assume that the condition \eqref{eq:cond} is fulfilled. We introduce $\xi=d-y.$ Then we have
\begin{equation}\label{eq:star}
 \omega(\xi + (x-d)) \leq \omega(\xi), \quad \xi \in [0,2d].
\end{equation}
Since $x-d>0$ the inequality \eqref{eq:star} implies the monotonicity of $\omega$ on $[0, 2d]$. Next, we set $\xi=2d$ in \eqref{eq:star}, thus obtaining
\begin{equation*}
 \omega(\eta) \leq \omega(2d)\quad\text{with}\quad \eta = x+d \geq 2d.
\end{equation*}

Conversely, assume that the condition (vii) is satisfied. Let $x\geq d$ and $y\in [0,2d]$ be arbitrary. If $x-y\in [0,2d]$, then the inequality $\omega(x-y) \leq \omega(d-y)$
follows from the monotonicity of $\omega$ and $x-y\geq d-y$. If $x-y>d$ we then obtain
\begin{equation*}
 \omega(x-y) \leq \omega(2d) \leq \omega(d-y).
\end{equation*}
\end{proof}


Our main results are as follows:

\begin{introtheorem}\label{introthm:1}
Under Assumptions \ref{ass:f} and \ref{ass:w}
there exists a bump solution $\widetilde{u}$ to the integro-diffe\-ren\-tial equation \eqref{eq:1}, that is, a stationary solution with $\supp f(\widetilde{u}(\cdot)-h)$ an interval.
Moreover,
\begin{equation}\label{ineq:pm}
u_-(x) \leq \widetilde{u}(x) \leq u_+(x)
\end{equation}
holds for all $x\in [-d,d]$ and, hence, the support of $f(\widetilde{u}(\cdot)-h)$ is contained in $[-d,d]$.
\end{introtheorem}

\begin{introtheorem}\label{introthm:2}
Assume in addition to Assumptions \ref{ass:f} and \ref{ass:w} that
\begin{itemize}
\item[(i)] $\omega\in W^{1,\infty}(\R)$, the Sobolev space of almost everywhere differentiable functions with essentially bounded derivative,
\item[(ii)] $\omega(x)\to 0$ as $|x|\to\infty$,
\item[(iii)] $f\in C^{1,\mu}(\R)$, that is, $f$ is continuously differentiable and its derivative is H\"{o}lder continuous with an exponent $\mu\in(0,1]$, $|f'(x)-f'(y)|\leq C |x-y|^\mu$.
\end{itemize}
Then the solution $\widetilde{u}$ referred to in Theorem \ref{introthm:1} belongs to $C_\infty(\R)$. It is a Lyapunov-unstable equilibrium of the integro-differential equation \eqref{eq:1}, that is, for all sufficiently small $\epsilon>0$ there is an initial value in the ball $B_\epsilon(\widetilde{u})\subset C_\infty(\R)$ such that the corresponding solution to \eqref{eq:1} leaves $B_\epsilon(\widetilde{u})$ in finite time.
\end{introtheorem}

Here $C_\infty(\R)$ denotes the Banach space of continuous functions vanishing at infinity.

It is straightforward to see that the conditions of Theorem \ref{introthm:2}are fulfilled for all three above examples of $\omega$ and for $f$ in \eqref{eq:f:ex} with $p>1.$
\subsection*{Acknowledgments} The authors thank H.-P.~Heinz for useful remarks. A.O.~is grateful to the Institute for Mathematics for its kind hospitality during her stay at the Johannes Gutenberg-Universit\"{a}t Mainz. Her work has been supported in part by the Deutsche Forschungsgemeinschaft, grant KO 2936/4-1, and by the Inneruniversit\"{a}ren For\-schungsf\"{o}rderung of the Johannes Gutenberg-Universit\"{a}t Mainz.

\section{Proof of Theorem \ref{introthm:1}}

In this section we treat $u_\pm$ defined in \eqref{eq:upm} as functions on $[-d,d]$. We define a nonlinear integral operator
\begin{equation}\label{eq:T}
(Tu)(x) := \int_{-d}^{d} \omega(x-y) f(u(y)-h) dy
\end{equation}
and consider the fixed point problem
\begin{equation*}
u=T u
\end{equation*}
in the real Banach space $C([-d,d])$. The cone
\begin{equation}\label{eq:K}
K:=\{u\in C([-d,d])\, :\, u(x)\geq 0\quad\text{for all}\quad x\in[-d,d]\}
\end{equation}
defines a partial order in $C([-d,d])$. We write $u\geq v$ if $u-v\in K$, $u>v$ if $u\geq v$ and $u\neq v$, and $u\gg v$ if $u-v$ is in the interior of $K$.

\begin{lemma}
Under Assumptions \ref{ass:f} and \ref{ass:w} the operator $T:\, C([-d,d]) \to C([-d,d])$ is monotone increasing and compact. Moreover,
$T u_- \ll u_-$ and $T u_+ \gg u_+$.
\end{lemma}

Recall that an operator $T$ acting on the ordered Banach space $X$ is called monotone increasing if $u\leq v$ implies $Tu\leq Tv$.

\begin{proof}
The linear integral operator
\begin{equation*}
u \mapsto \int_{-d}^d \omega(\cdot-y) u(y) dy
\end{equation*}
is continuous and compact as a mapping in $C([-d,d])$. Since the integral kernel $\omega(x-y)$ is positive for all $x,y\in[-d,d]$, it is monotone increasing.
The mapping $u\mapsto f(u-h)$ is continuous, monotone increasing, and bounded. This implies that $T$ is compact and monotone increasing.
Since $f(t)< \chi_{(0,\infty)}(t)$ on a set of positive measure, we obtain
\begin{equation*}
\begin{split}
(T u_-)(x) &= \int_{-d}^{d} \omega(x-y) f(u_-(y)-h) dy  <
\int_{-d}^{d} \omega(x-y) \chi_{(0,\infty)}(u_-(y)-h) dy\\ & = \int_{-\Delta_-}^{\Delta_-} \omega(x-y) dy = u_-(x),
\end{split}
\end{equation*}
which proves the first inequality. Similarly, the inequality $f(t)>\chi_{(\tau,\infty)}(t)$ holds on a set of positive measure.
Therefore,
\begin{equation*}
\begin{split}
(T u_+)(x) & = \int_{-d}^{d} \omega(x-y) f(u_+(y)-h) dy  >
\int_{-d}^{d} \omega(x-y) \chi_{(\tau,\infty)}(u_+(y)-h) dy\\ & = \int_{-\Delta_+}^{\Delta_+} \omega(x-y) dy = u_+(x),
\end{split}
\end{equation*}
which proves the second inequality.
\end{proof}

For any $u$ in the order interval $[u_-,u_+]:=\{u\in C([-d,d])\,:\, u_-\leq u\leq u_+\}$ we define the mapping
\begin{equation}\label{eq:T:hut}
(\widehat{T}u)(x) := \max\left\{\min\{(Tu)(x), u_+(x)\}, u_-(x) \right\},\qquad x\in[-d,d],
\end{equation}
or, more explicitly,
\begin{equation*}
(\widehat{T}u)(x) = \begin{cases} u_-(x) & \text{if}\quad (Tu)(x) \leq u_-(x), \\
(Tu)(x) & \text{if}\quad u_-(x) \leq (Tu)(x) \leq u_+(x),\\
u_+(x) & \text{if}\quad u_+(x) \leq (Tu)(x).\end{cases}
\end{equation*}
Since the r.h.s.\ in this definition is a continuous function satisfying $$u_-(x)\leq \max\left\{\min\{(Tu)(x), u_+(x)\}, u_-(x) \right\} \leq u_+(x)$$ for all $x\in[-d,d]$, $\widehat{T}$ is a self-mapping of $[u_-,u_+]$. Furthermore, $u_\pm$ are fixed points,
\begin{equation*}
\widehat{T} u_\pm = u_\pm.
\end{equation*}

\begin{lemma}\label{lem:epsilon}
The operator $\widehat{T}$ is monotone increasing and compact. Moreover, for sufficiently small $\epsilon>0$ one has
\begin{equation*}
\widehat{T}(u_-+\epsilon) \ll u_- +\epsilon \ll u_+ - \epsilon
\end{equation*}
and
\begin{equation*}
\widehat{T}(u_+-\epsilon) \gg u_+ -\epsilon \gg u_- + \epsilon.
\end{equation*}
\end{lemma}

\begin{proof}
By the monotonicity of the operator $T$ one has $T u_1\geq T u_2$ whenever $u_1\geq u_2$. Hence,
\begin{equation*}
\min\{(T u_1)(x), u_+(x)\} \geq \min\{(T u_2)(x), u_+(x)\}\quad\text{for all}\quad x\in[-d,d],
\end{equation*}
and, therefore,
\begin{equation*}
\max\left\{\min\{(T u_1)(x), u_+(x)\}, u_-(x)\right\} \geq \max\left\{\min\{(T u_2)(x), u_+(x)\},u_-(x)\right\}.
\end{equation*}
Thus, $\widehat{T}$ is monotone increasing.

Let $(u_n)$ be an arbitrary sequence in $[u_-,u_+]$. Since $T$ is compact, $(Tu_n)$ has a subsequence $(Tu_{n_k})$ converging to some $v\in C([-d,d])$. For arbitrary $\epsilon>0$ let $n_0\in\N$ be so large that
\begin{equation*}
|T u_{n_k}(x) - v(x)| < \epsilon\quad\text{for all}\quad k\geq n_0\quad\text{and}\quad x\in[-d,d].
\end{equation*}
Then one has
\begin{equation*}
\min\{T u_{n_k}(x), u_+(x)\} \leq \min\{v(x)+\epsilon, u_+(x)\}\leq \min\{v(x), u_+(x)\}+\epsilon
\end{equation*}
and
\begin{equation*}
\min\{T u_{n_k}(x), u_+(x)\} \geq \min\{v(x)-\epsilon, u_+(x)\}\geq \min\{v(x), u_+(x)\}-\epsilon,
\end{equation*}
which shows that $\min\{T u_{n_k}(x), u_+(x)\}$ converges uniformly to $\min\{v(x),u_+(x)\}$.
Similarly, one can show that $(\widehat{T}u_{n_k})$ converges uniformly to $\max\left\{\min\{v(x), u_+(x)\}, u_-(x) \right\}$, thus, proving that the range of $\widehat{T}$ is relatively compact.

Now assuming that the sequence $(u_n)$ converges to some $u\in [u_-,u_+]$ and using the continuity of $T$, we arrive at the conclusion that $(\widehat{T}u_n)$ converges to $(\widehat{T}u)$, thus proving that $\widehat{T}$ is continuous.

Since the mapping $u\in C([-d,d])\mapsto \inf_{x\in[-d,d]}u(x)$ is continuous, the functional $\rho: C([-d,d]) \to \R$,
\begin{equation*}
\rho(u) := \inf_{x\in[-d,d]} \left( u(x) - (Tu)(x)\right)
\end{equation*}
is continuous as well. Hence, due to $\rho(u_-)>0$, there is an $\epsilon>0$ such that $\rho(u)>0$ for all $u\in B_{2\epsilon}(u_-)$. We can choose $\epsilon$ so small that $u_-+\epsilon \ll u_+ -\epsilon$. Thus,
\begin{equation*}
T(u_-+\epsilon) \ll u_- +\epsilon \ll u_+ - \epsilon,
\end{equation*}
from which it follows that
\begin{equation*}
\min\{T(u_-+\epsilon),u_+\}=T(u_-+\epsilon)
\end{equation*}
and consequently
\begin{equation*}
\widehat{T}(u_-+\epsilon) = \max\{T(u_-+\epsilon), u_-\} \ll u_- + \epsilon.
\end{equation*}

The second inequality can be proven in the same way.
\end{proof}

The main tool for the proof of Theorem \ref{introthm:1} is Amann's theorem on three fixed points \cite[Theorem 14.2 and Corollary 14.3]{Amann} in the version of Zeidler \cite[Theorem 7.F and Corollary 7.40]{Zeidler}.

\begin{theorem}\label{thm:1.3}
Let $X$ be a real Banach space with an order cone having a nonempty interior. Assume there are four points in $X$
\begin{equation*}
p_1 \ll p_2 < p_3 \ll p_4
\end{equation*}
and a monotone increasing image compact operator $\widehat{T}:[p_1,p_4]\to X$ such that
 \begin{equation*}
\widehat{T}p_1 = p_1,\quad \widehat{T} p_2 < p_2,\quad \widehat{T} p_3 > p_3,\quad \widehat{T} p_4 = p_4.
\end{equation*}
Then $\widehat{T}$ has a third fixed point $p$ satisfying $p_1<p<p_4$, $p\notin[p_1,p_2]$, and $p\notin[p_3,p_4]$.
\end{theorem}

Recall that the operator $\widehat{T}$ is called image compact if it is continuous and its image $\widehat{T}[p_1,p_4]$ is relatively compact in $X$. In the case $X=C([-d,d])$, the order cone $K$ defined in \eqref{eq:K} is normal, that is, the order interval $[p_1,p_4]$ is norm bounded (see, e.g., \cite{Guo}). Therefore, the operator $\widehat{T}$ is image compact if and only if it is compact.

We choose $p_1=u_-$, $p_2=u_- + \epsilon$, $p_3=u_+-\epsilon$, $p_4=u_+$, where $\epsilon>0$ as in Lemma \ref{lem:epsilon}. Theorem \ref{thm:1.3} yields the existence of a fixed point $u_\ast$ of the operator $\widehat{T}$ satisfying $u_- \leq u_\ast \leq u_+$. Obviously, $u_\ast$ is a fixed point of the operator $T$ defined in \eqref{eq:T} as well.

\begin{lemma}\label{lem:bump}
If a fixed point $u$ of the operator $T$ satisfies the inequality $u(d)\leq u_+(d)=h$, then
\begin{equation*}
 \widetilde{u}(x)=\int_{-d}^{d} \omega(x-y) f(u(y)-h)dy,\quad x\in\R.
\end{equation*}
is a bump which solves \eqref{eq:1}.
\end{lemma}

\begin{proof}
Due to condition (vii) of Assumption \ref{ass:w} and Lemma \ref{lem:blin}, we have $\omega(x-y)\leq \omega(d-y)$ for all $x>d$. Hence,
$\widetilde{u}(x)\leq \widetilde{u}(d)\leq h$. This implies that $\widetilde{u}(x)$ solves the equation
 \begin{equation}\label{eq:utilde}
 \widetilde{u}(x)=\int_{-\infty}^{\infty} \omega(x-y) f(\widetilde{u}(y)-h)dy,\quad x\in\R.
\end{equation}
\end{proof}

\begin{remark}
We note that $\widetilde{u}$ is not an isolated solution of \eqref{eq:utilde}.
Indeed, $\widetilde{u}(\cdot-c)$ is again a solution for any $c\in\R$.
\end{remark}

\section{Proof of Theorem \ref{introthm:2}}

The proof of Theorem \ref{introthm:2} heavily relies on the Krein-Rutman theorem (see, e.g., \cite[Proposition 7.26]{Zeidler} or \cite[Theorem 6.1]{Krein:Rutman}):

\begin{theorem}
Let $X$ be a real Banach space with the order cone $K$ having a nonempty interior. Suppose that $T:X \to X$ is linear, compact, and positive, with the spectral radius $r(T)>0$. Then $r(T)$ is an eigenvalue of $T$ with all eigenvectors in $K$.
\end{theorem}

The second tool is a classical result on the instability of equilibrium solutions of differential equations
\cite[Theorem VII.2.3]{Daleckij:Krein} (\textit{cf.} also Corollary 5.1.6 in \cite{Henry}).

\begin{theorem}\label{thm:Krein}
Let $X$ be a Banach space, $A$ be a linear continuous operator on $X$, $F: X\to X$ a nonlinear Lipschitz continuous operator. If
\begin{itemize}
\item[(i)] $v_0=0$ satisfies $Av_0+Fv_0=0$,
\item[(ii)] the operator $F$ obeys the estimate
\begin{equation}\label{eq:est}
\|Fv\|\leq C \|v\|^{1+\mu},\qquad C>0, \qquad \mu>0
\end{equation}
for all $u\in X$ with $\|v\|<\epsilon$ for some $\epsilon>0$,
\item[(iii)] the spectrum $\sigma(A)$ contains a point $\lambda$ with $\Re\lambda>0$,
\end{itemize}
then $v_0$ is an unstable equilibrium of the differential equation
\begin{equation*}
v_t = Av + Fv,\qquad t>0.
\end{equation*}
\end{theorem}

Let $u_\ast\in C([-d,d])$ denote the fixed point of the operator $T$ \eqref{eq:T} referred to in the previous section. From the condition (i)
of Theorem \ref{introthm:2} it follows that $u_\ast$ belongs to $C^1([-d,d])$. For the proof see Lemma \ref{lem:app:2} in Appendix.
Due to \eqref{ineq:pm}, one has $u_\ast(-d)\leq h$, $u_\ast(0)\geq u_-(0)>h$, and $u_\ast(d)\leq h$. Thus, $u_\ast$ is not monotone.

We observe that under the conditions of Theorem \ref{introthm:2} the operator $T$ is Fr\'{e}chet differentiable with
\begin{equation*}
(T^\prime(u) v) (x) = \int_{-d}^d \omega(x-y) f'(u(y)-h) v(y) dy,\qquad v\in C([-d,d]).
\end{equation*}
It is a linear, compact, and positive operator with respect to the cone defined by \eqref{eq:K}.

Since $u_\ast(\pm d)\leq h$, integrating by parts we obtain
\begin{equation*}
\begin{split}
u'_\ast(x) &= \int_{-d}^d \omega'(x-y) f(u_\ast(y)-h) dy = -\int_{-d}^d \frac{\partial}{\partial y} \omega(x-y) f(u_\ast(y)-h) dy \\
&= \int_{-d}^d \omega(x-y) f'(u_\ast(y)-h) u'_\ast(y) dy.
\end{split}
\end{equation*}
Hence, $u'_\ast$ is an eigenfunction of the operator $T'(u_\ast)$ with eigenvalue $1$. Thus, the spectral radius $r(T'(u_\ast))$ is not smaller than $1$.

Assume that $r(T'(u_\ast))=1$. Applying the Krein-Rutman theorem with $X=C([-d,d])$, the cone $K$ defined in \eqref{eq:K}, and the operator $T'(u_\ast)$, we obtain that $u'_\ast(x)\geq 0$ for all $x\in[-d,d]$, which is a contradiction. Thus, $r(T'(u_\ast))>1$. Again by the Krein-Rutman theorem $r(T'(u_\ast))$ is an eigenvalue of $T'(u_\ast)$. Hence, we arrive at the conclusion that the Fr\'{e}chet derivative of $u\mapsto -u + Tu$ at the point $u_\ast$ has a strictly positive eigenvalue.

Denote by $\widetilde{T}$ the nonlinear integral operator defined via
\begin{equation*}
(\widetilde{T} u)(x) := \int_\R \omega(x-y) f(u(y)-h) dy.
\end{equation*}
Observe that under the condition (ii) of Theorem \ref{introthm:2}, $\widetilde{T}$ maps $C_\infty(\R)$ into itself.
Hence, the bump $\widetilde{u}$ referred to in Theorem \ref{introthm:1} belongs to $C_\infty(\R)$.

\begin{lemma}\label{lem:spec}
Let the conditions of Theorem \ref{introthm:2} be satisfied. Then the Fr\'{e}chet derivative $\widetilde{T}'(\widetilde{u})\, :\, C_\infty(\R)\to C_\infty(\R)$ of the operator $\widetilde{T}$,
\begin{equation*}
\left(\widetilde{T}'(\widetilde{u}) v\right) (x) = \int_\R \omega(x-y) f'(\widetilde{u}(y)-h) v(y) dy,
\end{equation*}
is compact.
\end{lemma}

\begin{proof}
The proof is based on the following compactness criterion \cite[Theorem IV.6.5]{Dunford:Schwartz}:
\begin{itemize}
 \item A bounded subset $S\subset C(\R)$ is relatively compact if and only if for every $\epsilon>0$ there is a finite collection of sets $E_i\subset\R$, $i=1,\ldots,n$, $\bigcup_{i=1}^n E_i=\R$, and points $x_i\in E_i$ such that
\begin{equation*}
\sup_{\varphi\in S}\sup_{x\in E_i} |\varphi(x_i)-\varphi(x)|<\epsilon
\end{equation*}
for all $i=1,\ldots,n$.
\end{itemize}

Consider a set
\begin{equation*}
 S:= \{\varphi\in C_\infty(\R)\, :\, \varphi=\widetilde{T}'(\widetilde{u})v,\quad v\in B_1(0)\subset C_\infty(\R) \}.
\end{equation*}
Using the mean value theorem we obtain
\begin{equation*}
\begin{split}
|\varphi(x)| &\leq \int_{-d}^d |\omega(x-y)| f'(\widetilde{u}(y)-h) dy\\
& \leq |\omega(x-\eta)|\cdot f'(\widetilde{u}(\eta)-h)
\end{split}
\end{equation*}
for some $\eta\in[-d,d]$ and any $\varphi\in S$. Hence,
\begin{equation*}
|\varphi(x)| \leq C \sup_{y\in[-d,d]} |\omega(x-y)|
\end{equation*}
for all $\varphi\in S$. Therefore, by the condition (ii) of Theorem \ref{introthm:2}, for an arbitrary $\epsilon>0$ we can choose $R>d$ so large that
\begin{equation*}
\sup_{\varphi\in S} |\varphi(x)| < \epsilon/2\quad\text{for all}\quad |x|>R.
\end{equation*}
Thus, we obtain
\begin{equation}\label{eq:11}
\sup_{\varphi\in S} \sup_{x\in E_\pm} \left| \varphi(x) - \varphi(\pm 2R)\right|<\epsilon,
\end{equation}
where $E_\pm := \{x\in\R\, |\, \pm x > \pm R\}$.

Let $S_0$ be the set in $C([-R,R])$ consisting of all functions in $S$ restricted to the interval $[-R,R]$,
\begin{equation*}
 S_0 := \{\varphi_0= \varphi|_{[-R,R]}\, : \, \varphi\in S\}.
\end{equation*}
This set is the range of the compact integral operator
\begin{equation*}
 v \mapsto \int_{-R}^R \omega(\cdot-y) f'(\widetilde{u}(y)-h) v(y) dy,
\end{equation*}
mapping $C([-R,R])$ into itself. Thus, $S_0$ is relative compact.

By the compactness criterion above, there is a finite collection $(E_i)_{i=1}^n$ of subsets in $[-R,R]$ and points $x_i\in E_i$ such that
\begin{equation*}
\sup_{\varphi\in S}\sup_{x\in E_i} |\varphi(x_i)-\varphi(x)|<\epsilon
\end{equation*}
for all $i=1,\ldots,n$. Combining this with \eqref{eq:11}, we arrive at the conclusion that the collection $(E_1,\ldots,E_n, E_+, E_-)$ with points $(x_1,\ldots,x_n,2R,-2R)$ satisfies the condition of the compactness criterion, thus, proving that $S$ is a relative compact set.
Hence, $\widetilde{T}'(\widetilde{u})$ is a compact operator.
\end{proof}

Now we show that the linear operators $\widetilde{T}'(\widetilde{u})$ and $T'(u_\ast)$ have the same spectra. Since both operators are compact, it suffices to prove that they have the same eigenvalues. Assume that $\lambda\neq 0$ is an eigenvalue of $T'(u_\ast)$ with an eigenfunction $v\in C([-d,d])$. We set
\begin{equation*}
\widetilde{v}(x):=\frac{1}{\lambda} \int_{-d}^d \omega(x-y) f'(u_\ast(y)-h) v(y) dy,\quad x\in\R.
\end{equation*}
It is easy to check that $\widetilde{v}$ is an eigenfunction of $\widetilde{T}'(\widetilde{u})$ corresponding to the eigenvalue $\lambda$.
Conversely, assume that $\widetilde{v}\in C_\infty(\R)$ is an eigenfunction of $\widetilde{T}'(\widetilde{u})$ corresponding to the eigenvalue $\lambda\neq 0$. Then
\begin{equation*}
\lambda \widetilde{v}(x) = \int_\R \omega(x-y) f'(\widetilde{u}(y)-h) \widetilde{v}(y) dy = \int_{-d}^d \omega(x-y) f'(u_\ast(y)-h) \widetilde{v}(y) dy
\end{equation*}
holds for all $x\in[-d,d]$. This implies that $\lambda$ is an eigenvalue of $T'(u_\ast)$ with an eigenfunction $v:=\widetilde{v}|_{[-d,d]}$.

We arrive at the conclusion that the linear operator $\widetilde{T}'(\widetilde{u})$ has an eigenvalue $\lambda>1$, and, thus, the condition (iii) of Theorem \ref{thm:Krein} is fulfilled.

Plugging $u(x,t)=\widetilde{u}(x) + w(x,t)$ into the equation $u_t=-u+\widetilde{T}u$ we obtain
\begin{equation*}
w_t = A w + Fw,
\end{equation*}
where
\begin{equation*}
A v = -v + \widetilde{T}'(\widetilde{u})v\qquad\text{and}\qquad Fv= \widetilde{T}(\widetilde{u}+v) - \widetilde{T} \widetilde{u} - \widetilde{T}'(\widetilde{u})v
\end{equation*}
for any $v\in C_\infty(\R)$.

{}From the continuous differentiability of $f$ it easily follows that $\widetilde{T}$ is Lipschitz continuous.
By the mean value theorem one has
\begin{equation*}
\begin{split}
\widetilde{T}(\widetilde{u}+v)(x) - (\widetilde{T} \widetilde{u})(x) &= \int_\R \omega(x-y) \left(f(\widetilde{u}(y)+v(y)-h)-f(\widetilde{u}(y)-h)\right) dy\\
=& \int_\R \omega(x-y) f'(a(y)-h) v(y) dy,
\end{split}
\end{equation*}
where $a(y)$ is a point between $\widetilde{u}(y)$ and $\widetilde{u}(y)+v(y)$, $y\in\R$. Hence,
\begin{equation*}
\begin{split}
& \widetilde{T}(\widetilde{u}+v)(x) - (\widetilde{T} \widetilde{u})(x) - (\widetilde{T}'(\widetilde{u})v)(x) \\ & = \int_\R \omega(x-y) \left(f'(a(y)-h)-f'(\widetilde{u}(y)-h)\right) v(y) dy.
\end{split}
\end{equation*}
{}From the H\"{o}lder continuity of $f'$ it follows that
\begin{equation*}
|f'(a(y)-h)-f'(\widetilde{u}(y)-h)| \leq C |a(y)-\widetilde{u}(y)|^\mu \leq C |v(y)|^\mu.
\end{equation*}
Thus,
\begin{equation*}
\begin{split}
|\widetilde{T}(\widetilde{u}+v)(x) - (\widetilde{T} \widetilde{u})(x) - (\widetilde{T}'(\widetilde{u})v)(x)| &\leq C \int_\R |\omega(x-y)| |v(y)|^{1+\mu} dy\\
& \leq C \|v\|_\infty^{1+\mu} \|\omega\|_{L^1(\R)},
\end{split}
\end{equation*}
which implies that the condition (ii) of Theorem \ref{thm:Krein} is fulfilled.
Consequently, by Theorem \ref{thm:Krein} with $X=C_\infty(\R)$, $\widetilde{u}$ is an unstable equilibrium of the equation
\begin{equation}\label{eq:zeit}
u_t = -u + \widetilde{T}u.
\end{equation}
This completes the proof of Theorem \ref{introthm:2}.

\appendix
\section*{Appendix}
\setcounter{lem}{0}

\begin{lem}\label{lem:app:1}
Under Assumption \ref{ass:w} the function $u_-$ given by \eqref{eq:upm} is a stationary solution of \eqref{eq:1} with $f=\chi_{(0,\infty)}$. Similarly,
$u_+$ is a solution of \eqref{eq:1} with $f=\chi_{(\tau,\infty)}$. In particular, $\supp \chi_{(0,\infty)}(u_-(\cdot)-h)=[-\Delta_-,\Delta_-]$
and $\supp \chi_{(\tau,\infty)}(u_+(\cdot)-h)=[-\Delta_+,\Delta_+]$.
\end{lem}

\begin{proof}
To prove that $u_-$ solves \eqref{eq:1} with $f=\chi_{(0,\infty)}$ it suffices to show that
\begin{equation*}
 u_-(x)\geq h\quad \text{for all}\quad x\in[0,\Delta_-]
\end{equation*}
and
\begin{equation*}
 u_-(x)< h\quad \text{for all}\quad x\in(\Delta_-,\infty).
\end{equation*}

If $x\in[0,\Delta_-]$ we represent $u_-$ as follows
\begin{equation*}
\begin{split}
 u_-(x) &= \int_{x-\Delta_-}^{x+\Delta_-} \omega(z) dz\\
&= \int_0^{2\Delta_-} \omega(z) dz - \int_{x+\Delta_-}^{2\Delta_-} \omega(z) dz + \int_{x-\Delta_-}^0 \omega(z) dz.
\end{split}
\end{equation*}
Observe that the first integral equals $h$. Using the symmetry of $\omega(z)$ we, thus, obtain
\begin{equation*}
 u_-(x) = h + \int_0^{\Delta_- -x }(\omega(z)-\omega(z+x+\Delta_-)) dz \geq h
\end{equation*}
by the condition (vii) of Assumption \ref{ass:w}.

If $x>\Delta_-$ we represent $u_-$ as
\begin{equation*}
 u_-(x) = \int_{I_1} \omega(x-y) dy + \int_{I_2} \omega(x-y) dy,
\end{equation*}
where
\begin{equation*}
 I_1:=\{y\in [-\Delta_-,\Delta_-]\, :\, 0< x-y < 2d\}
\end{equation*}
and
\begin{equation*}
 I_2:=\{y\in [-\Delta_-,\Delta_-]\, :\, 2d< x-y\}.
\end{equation*}
For any $y\in I_1$ we have $\omega(x-y) < \omega(\Delta_- -y)$ since $\omega$ is decreasing on $[0,2d]$ by the condition (vii) of Assumption \ref{ass:w}.
If $y\in I_2$, then again by the condition (vii) of Assumption \ref{ass:w} we obtain the inequality
\begin{equation*}
 \omega(x-y) < \omega(2d) < \omega(\Delta_- - y).
\end{equation*}
Hence, in both cases the inequality
\begin{equation*}
 u_-(x) < \int_{-\Delta_-}^{\Delta_-} \omega(\Delta_- - y)dy = h
\end{equation*}
is valid.

That $u_+$ is a solution to \eqref{eq:1} with $f=\chi_{(\tau,\infty)}$ can be proved in the same way.
\end{proof}

\begin{lem}\label{lem:app:2}
 If the integral kernel $\omega$ satisfies Assumption \ref{ass:w} and the condition (i) of Theorem \ref{introthm:2}, then the stationary solution $\widetilde{u}$ referred to in Theorem \ref{introthm:1} is continuously differentiable.
\end{lem}

\begin{proof}
{}From \eqref{eq:utilde} it follows that the derivative $\widetilde{u}'(x)$ exists for each $x\in\R$ and is given by
\begin{equation*}
 \widetilde{u}'(x) = \int_{\R} \omega'(x-y) f(\widetilde{u}(y)-h) dy =\int_{\R} \omega'(y) f(\widetilde{u}(x-y)-h) dy.
\end{equation*}
Hence, since $\|\omega'\|_{L^\infty(\R)}<\infty$, for any $\delta>0$ we have
\begin{equation*}
\begin{split}
 |\widetilde{u}'(x+\delta)-\widetilde{u}'(x)| & \leq \int_{\R} |\omega'(y)| |f(\widetilde{u}(x+\delta-y)-h)-f(\widetilde{u}(x-y)-h)| dy\\
& \leq \|\omega'\|_{L^\infty(\R)} \int_\R |f(\widetilde{u}(x+\delta-y)-h)-f(\widetilde{u}(x-y)-h)| dy\\
& = \|\omega'\|_{L^\infty(\R)} \|f(\widetilde{u}(\cdot+\delta)-h)-f(\widetilde{u}(\cdot)-h)\|_{L^1(\R)}.
\end{split}
\end{equation*}
By the continuity of translations in $L^1(\R)$ we obtain that $|\widetilde{u}'(x+\delta)-\widetilde{u}'(x)|\to 0$ as $\delta\to 0$, thus proving
that $\widetilde{u}'\in C(\R)$.
\end{proof}


\end{document}